\documentclass[12pt]{article}

\usepackage{amsfonts,amssymb,amsthm, amsmath}

\begin{document}


\title{Lipschitz shadowing implies structural stability}

\author{Sergei Yu. Pilyugin \footnote{$^1$ Faculty of Mathematics and Mechanics, St.\ Petersburg
State University,
                  University av.\ 28,
                  198504, St.\ Petersburg, Russia. sp@sp1196.spb.edu} and Sergey Tikhomirov\footnote{Department of Mathematics, National Taiwan University, No.
1, Section 4, Roosevelt Road, Taipei 106, Taiwan.
sergey.tikhomirov@gmail.com}}

\date{}

\newtheorem{theorem}{Theorem}
\newtheorem{corollary}[theorem]{Corollary}
\newtheorem{proposition}[theorem]{Proposition}
\newtheorem{lemma}[theorem]{Lemma}
\theoremstyle{definition}
\newtheorem{remark}[theorem]{Remark}

\newcommand{\R}{\ensuremath{\mathbb{R}}}
\newcommand{\N}{\ensuremath{\mathbb{N}}}
\newcommand{\Z}{\ensuremath{\mathbb{Z}}}

\newcommand{\ep}{\varepsilon}
\newcommand{\lam}{\lambda}
\newcommand{\tm}{T_{p_k}M}

\newcommand{\LL}{\mathcal{L}}
\newcommand{\ek}{\exp_{p_k}}
\newcommand{\emk}{\exp^{-1}_{p_{k+1}}}
\newcommand{\ekk}{\exp_{p_{k+1}}}
\newcommand{\Mane}{\mbox{Ma$\tilde{\mbox{n}}\acute{\mbox{e}}$}}

\maketitle

\begin{abstract}
We show that Lipschitz shadowing property of a diffeomorphism is
equivalent to structural stability. As a corollary, we show that an
expansive diffeomorphism having Lipschitz shadowing property is
Anosov.
\end{abstract}
\vspace{2pc}
 \maketitle

\section{Introduction}

The theory of shadowing of approximate trajectories
(pseudotrajectories) of dynamical systems is now a well developed
part of the global theory of dynamical systems (see, for example,
the monographs \cite{PilBook, PalmBook}).

This theory is closely related to the classical theory of structural
stability. It is well known that a diffeomorphism has shadowing
property in a neighborhood of a hyberbolic set \cite{Ano, Bow} and a
structurally stable diffeomorphism has shadowing property on the
whole manifold \cite{Rob, Mor, Saw}. Analyzing the proofs of the
first shadowing results by Anosov \cite{Ano} and Bowen \cite{Bow},
it is easy to see that, in a neighborhood of a hyperbolic set, the
shadowing property is Lipschitz (and the same holds in the case of a
structurally stable diffeomorphism, see \cite{PilBook}).

At the same time, it is easy to give an example of a diffeomorphism
that is not structurally stable but has shadowing property (see
\cite{PilVar}, for example).

Thus, structural stability is not equivalent to shadowing.

One of possible approaches in the study of relations between
shadowing and structural stability is the passage to
$C^1$-interiors. At present, it is known that the $C^1$-interior of
the set of diffeomorphisms having shadowing property coincides with
the set of structurally stable diffeomorphisms \cite{Sak}. Later, a
similar result was obtained for orbital shadowing property (see
\cite{PilRodSak} for details).

Here, we are interested in the study of the above-mentioned
relations without the passage to $C^1$-interiors. Let us mention in
this context that Abdenur and Diaz conjectured that a $C^1$-generic
diffeomorphism with the shadowing property is structurally stable;
they have proved this conjecture for so-called tame diffeomorphisms
\cite{AbdDiaz}. Recently, the first author has proved that the
so-called variational shadowing is equivalent to structural
stability \cite{PilVar}.

In this short note, we show that Lipschitz shadowing
property is equivalent to structural stability.

As a corollary, we show that an expansive diffeomorphism having
Lipschitz shadowing property is Anosov. Let us mention that Ombach
\cite{Omb} and Walters \cite{Walt} showed that a diffeomorphism $f$
is Anosov if and only if $f$ has shadowing property and is strongly
expansive (which means that all the diffeomorphisms in a $C^1$-small
neighborhood of $f$ are expansive with the same expansivity
constant).

In addition, let us mention the recent paper \cite{OsipPilTikh},
where it is shown that Lipschitz periodic shadowing is equivalent to
$\Omega$-stability.

\section{Main result}

Let us pass to exact definitions and statements.

Let $f$ be a diffeomorphism of class $C^1$ of an $m$-dimensional closed
smooth manifold $M$ with Riemannian metric $\mbox{dist}$.
Let $Df(x)$ be the differential of $f$ at a point $x$.
For a point $p\in M$, we denote $p_k=f^k(p),k\in{\Z}$.

Denote by $T_xM$ the tangent space of $M$ at a point $x$;
let $|v|,\,v\in T_xM$, be the norm of $v$ generated by
the metric $\mbox{dist}$.

We say that $f$ has {\em shadowing property}
if for any $\epsilon>0$ there exists $d>0$ with the
following property: for any sequence of points $X=\{x_k\in M\}$
such that
$$
\mbox{dist}(x_{k+1},f(x_k))<d,\quad k\in{\Z},
$$
there exists a point $p\in M$ such that
\begin{equation}
\label{2}
\mbox{dist}(x_{k},p_k)<\epsilon,\quad k\in{\Z}
\end{equation}
(if inequalities (\ref{2}) hold, one says that the trajectory
$\{p_k\}$ $\epsilon$-{\em shadows} the $d$-pseudo\-trajectory $X$).

We say that $f$ has {\em Lipschitz  shadowing property} if there
exist constants $\mathcal{L},d_0>0$ with the following property: For
any $d$-pseudotrajectory $X$ as above with $d\leq d_0$ there exists
a point $p\in M$ such that
\begin{equation}
\label{3}
\mbox{dist}(x_{k},p_k)\leq \LL d,\quad k\in{\Z}.
\end{equation}

The main result of this note is the following statement.

\begin{theorem}\label{thm1}
The following two statements are equivalent:
\begin{itemize}
\item[(1)] $f$ has Lipschitz  shadowing property;
\item[(2)] $f$ is structurally stable.
\end{itemize}
\end{theorem}

\begin{remark}\label{rem1}
Let us recall that a diffeomorphism $f$ is called {\em expansive} if
there exists a positive number $a$ (expansivity constant) such that
if $x,y\in M$ and  $$\mbox{dist}(f^k(x),f^k(y))\leq a$$ for all
$k\in\Z$, then $x=y$. The above theorem has the following corollary.
\end{remark}

\begin{corollary}\label{cor1}
The following two statements are equivalent:
\begin{itemize}
\item[(1)] $f$ is expansive and has Lipschitz  shadowing property;
\item[(2)] $f$ is Anosov.
\end{itemize}
\end{corollary}

\begin{proof}[Proof of the corollary.] The implication $(2)\Rightarrow
(1)$ is well known (see, for example, \cite{PilBook}). By our
theorem, condition (1) of the corollary implies that $f$ is
structurally stable, and it was shown by $\Mane$ that an expansive
structurally stable diffeomorphism is Anosov (see \cite{Mane1}).
\end{proof}

Now we pass to the proof of the main theorem.

The implication $(2)\Rightarrow (1)$ is well known (see, for
example, \cite{PilBook}).

In the proof of the implication $(1)\Rightarrow (2)$, we use the
following two known results (Propositions \ref{Prop1} and
\ref{Prop2}).

First we introduce some notation. For a point $p\in M$,
define the following two subspaces of $T_pM$:
$$
B^+(p)=\{v\in T_pM\,:\,|Df^k(p)v|\to 0,\quad k\to +\infty\}
$$
and
$$
B^-(p)=\{v\in T_pM\,:\,|Df^k(p)v|\to 0,\quad k\to -\infty\}.
$$

\begin{proposition}\label{Prop1}
[\Mane, \cite{Mane2}]. The diffeomorphism $f$ is structurally stable
if and only if
$$
B^+(p)+B^-(p)=T_pM
$$
for any $p\in M$.
\end{proposition}

Consider a sequence of linear isomorphisms
$$
\mathcal{A}=\{A_k:{\R}^m\to{\R}^m,\;k\in{\Z}\}
$$
for which there exists a constant $N>0$ such that
$\|A_k\|,\|A^{-1}_k\|\leq N$.
Fix two indices $k,l\in{\Z}$ and
denote
$$
\Phi(k,l)=
\begin{cases}
A_{k-1}\circ\dots\circ A_l, & \mbox{$l<k;$}\\
 \mbox{Id}, &
\mbox{$l=k;$}\\ A^{-1}_{k}\circ\dots\circ A^{-1}_{l-1}, &
\mbox{$l>k.$}\\
\end{cases}
$$
Set
$$
B^+(\mathcal{A})=\{v\in {\R}^m\,:\,|\Phi(k,0)v|\to 0,\quad k\to
+\infty\}
$$
and
$$
B^-(\mathcal{A})=\{v\in {\R}^m\,:\,|\Phi(k,0)v|\to 0,\quad k\to
-\infty\}.
$$

\begin{proposition}\label{Prop2}[Pliss, \cite{Pli}] The following two statements are
equivalent:
\begin{itemize}
\item[(a)] For any bounded sequence $\{w_k\in{\R}^m,\;k\in{\Z}\}$
there exists a bounded sequence $\{v_k\in{\R}^m,\;k\in{\Z}\}$ such
that
$$
v_{k+1}=A_kv_k+w_k,\quad k\in{\Z};
$$

\item[(b)] the sequence $\mathcal{A}$ is hyperbolic on any of the
rays $[0,+\infty)$ and $(-\infty,0]$ (see the definition in
\cite{PilGen}), and the spaces $B^+(\mathcal{A})$ and
$B^-(\mathcal{A})$ are transverse.
\end{itemize}
\end{proposition}

\begin{remark}\label{rem2}
In fact, Pliss considered in \cite{Pli} not sequences of
isomorphisms but homogeneous linear systems of differential
equations; the relation between these two settings is discussed in
\cite{PilGen}.
\end{remark}

\begin{remark}\label{rem3}
Later, a  statement analogous to Proposition \ref{Prop2} was proved
by Palmer \cite{Palm1, Palm2}; Palmer also described Fredholm
properties of the corresponding operator
$$
\{v_k\in\R^m:\,k\in\Z\}\mapsto\{v_k-A_{k-1}v_{k-1}\}.
$$
\end{remark}

We fix a point $p\in M$ and consider the isomorphisms
$$
A_k=Df(p_k):T_{p_k}M\to T_{p_{k+1}}M.
$$
Note that the implication (a)$\Rightarrow$(b) of Proposition
\ref{Prop2} is valid for the sequence $\{A_k\}$ with ${\R}^m$
replaced by $T_{p_k}M$.

It follows from Propositions \ref{Prop1} and \ref{Prop2} that our
main theorem is a corollary of the following statement (indeed, we
prove that Lipschitz shadowing property implies the validity of
statement (a) of Proposition \ref{Prop2} for any trajectory
$\{p_k\}$ of $f$, while the validity of statement (b) of this
proposition implies the structural stability of $f$ by Proposition
\ref{Prop1}).

\begin{lemma}\label{lem1}
If $f$ has the Lipschitz shadowing property with constants
$\LL,d_0$, then for any sequence $\{w_k\in \tm,k \in \Z\}$ such that
$|w_k|< 1, k\in\Z,$ there exists a sequence $\{v_k\in \tm,k \in
\Z\}$ such that
\begin{equation}
\label{3.5}
|v_k| \leq 8\LL + 1, \quad v_{k+1} = A_k v_k + w_k, \quad k \in \Z.
\end{equation}
\end{lemma}

To prove Lemma \ref{lem1}, we first prove the following statement.

\begin{lemma}\label{lem2}
Assume that $f$ has the Lipschitz shadowing property with constants
$\LL,d_0$. Fix a trajectory $\{p_k\}$ and a natural number $n$. For
any sequence $\{w_k\in\tm,k \in [-n,n]\}$ such that $|w_k|< 1$ for
$k \in [-n,n]$ and $w_k=0$ for $k \notin [-n,n]$ there exists a
sequence $\{z_k\in\tm,k \in \Z\}$ such that
\begin{equation}
\label{1.4}
|z_k| \leq 8\LL + 1, \quad k \in \Z,
\end{equation}
and
\begin{equation}
\label{1.5}
z_{k+1} = A_k z_k + w_k, \quad k \in [-n, n].
\end{equation}
\end{lemma}

\begin{proof}

First we locally ``linearize" the diffeomorphism $f$ in a
neighborhood of the trajectory $\{p_k\}$.

Let $\exp$ be the standard exponential mapping on the tangent bundle
of $M$ and let $\exp_x: T_xM \to M$ be the corresponding exponential
mapping at a point $x$.

We introduce the mappings
$$
F_k=\emk\circ f\circ\ek: T_{p_k}M\to T_{p_{k+1}}M.
$$
It follows from the standard properties of the exponential mapping
that $D\exp_x(0)=\mbox{Id}$; hence, $DF_k(0)=A_k. $ Since $M$ is
compact, for any $\mu>0$ we can find $\delta>0$ such that if
$|v|\leq\delta$, then
\begin{equation}
\label{mal}
|F_k(v)-A_kv|\leq\mu|v|.
\end{equation}

Denote by $B(r,x)$ the ball in $M$ of radius $r$ centered
at a point $x$ and by $B_T(r,x)$ the ball in $T_xM$ of radius $r$ centered
at the origin.

There exists $r>0$ such that, for any $x\in M$, $\exp_x$ is a diffeomorphism
of $B_T(r,x)$ onto its image, and $\exp^{-1}_x$ is a diffeomorphism
of $B(r,x)$ onto its image. In addition, we may assume that $r$ has
the following property.

If $v,w\in B_T(r,x)$, then
\begin{equation}
\label{eq1}
\frac{\mbox{dist}(\exp_x(v),\exp_x(w))}{|v-w|}\leq 2;
\end{equation}
if $y,z\in B(r,x)$, then
\begin{equation}
\label{eq2}
\frac{|\exp^{-1}_x(y)-\exp^{-1}_x(z)|}{\mbox{dist}(y,z)}\leq 2.
\end{equation}

Now we pass to construction of pseudotrajectories; every time,
we take $d$ so small that the considered points of our pseudotrajectories,
points of shadowing trajectories, their ``lifts" to tangent spaces
etc belong to the corresponding balls $B(r,p_k)$ and $B_T(r,p_k)$
(and we do not repeat this condition on the smallness of $d$).

Fix a sequence $w_k$ having the properties stated in Lemma
\ref{lem2}. Consider the sequence $\{\Delta_k\in \tm,k\in [-n,
n+1]\}$ defined as follows:
\begin{equation}
\label{2.1} \begin{cases}
\Delta_{-n} = 0, \\
\Delta_{k+1} = A_k \Delta_k + w_k, \quad k \in [-n, n]. \end{cases}
\end{equation}
Let $Q=\max_{k\in[-n,n+1]}|\Delta_k|$.

Fix a small $d>0$ and construct a pseudotrajectory $\{\xi_k\}$
as follows:
$$
\begin{cases}
\xi_k=\ek(d\Delta_k), & \mbox{$k\in[-n,n+1],$} \\
\xi_l=f^{l+n}(\xi_{-n}), & \mbox{$l\leq -n-1,$}\\
\xi_l=f^{l-n-1}(\xi_{n+1}), & \mbox{$l> n+1.$}
\end{cases}
$$
Note that definition (\ref{2.1}) of the vectors $\Delta_k$ and
condition (\ref{eq1}) imply that if $d$ is small enough, then the
following inequality holds:
$$
\mbox{dist}(\xi_{k+1}, \exp_{p_{k+1}}(dA_k \Delta_k)) < 2d.
$$
Since
$$
f(\xi_k)=\exp_{p_{k+1}}(F_k(d\Delta_k)),
$$
condition (\ref{mal}) with $\mu<1$ implies that if $d$ is small
enough, then
$$
\mbox{dist}(\exp_{p_{k+1}}(dA_k \Delta_k), f(\xi_k)) < 2d.
$$
Hence,
$$
\mbox{dist}(f(\xi_k),\xi_{k+1})\leq 4d.
$$

Let us note that the required smallness of $d$ is determined
by the chosen trajectory $\{p_k\}$, the sequence $w_k$, and the number $n$.

The Lipschitz shadowing property of $f$ implies that if $d$ is small
enough, then there exists an exact trajectory $\{y_k\}$ such that
\begin{equation}
\label{2.5}
\mbox{dist}(\xi_k,y_k) \leq 4\LL d, \quad k \in [-n, n+1].
\end{equation}

Consider the finite sequence
$$
\{t_k =
\frac{1}{d}\mbox{exp}_{p_k}^{-1}(y_k),\;k \in [-n, n+1]\}.
$$
Inequalities (\ref{2.5}) and (\ref{eq2})
imply that
\begin{equation}
\label{Add7.1.1}
|\Delta_k - t_k| < 8\LL.
\end{equation}
Consider the finite sequence $\{b_k \in T_{p_k}M,\;k \in [-n,
n+1]\}$ defined as follows:
\begin{equation}
\label{fly2}
b_{-n} = t_{-n}, \quad b_{k+1} = A_k b_k, \quad k \in [-n, n].
\end{equation}
Obviously, the following inequalities hold for $k \in [-n, n+1]$:
$$
\mbox{dist}(y_k, p_k) \leq \mbox{dist}(y_k, \xi_k) +\mbox{dist}(p_k, \xi_k) \leq 4\LL
d + 2 d |\Delta_k| \leq 2(Q + 2\LL)d.
$$
These inequalities and inequalities (\ref{eq2}) imply that
\begin{equation}
\label{tk}
|t_k|\leq 4(Q+2\LL).
\end{equation}

Take $\mu_1>0$ such that
\begin{equation}
\label{mu}
((N+1)^{2n} + (N+1)^{2n-1}+ \dots + 1)\mu_1<1,
\end{equation}
where $N = \sup\|A_k\|$.

Set
$$
\mu=\frac{\mu_1}{4(Q+2\LL)}
$$
and
consider $d$ so small that inequality (\ref{mal}) holds for
$\delta = 4(Q + 2L)d$.

The definition of the vectors $t_k$ implies that
$dt_{k+1}=F_k(dt_k)$; since
$$
|dt_k|\leq 4d(Q+2\LL)
$$
by (\ref{tk}),
we deduce from estimate (\ref{mal}) applied to $v=dt_k$
that
$$
|dt_{k+1}-dA_kt_k|\leq \mu d|t_k|.
$$
Now we deduce from inequalities (\ref{tk}) that
\begin{equation}
\label{mu1}
|t_{k+1} - A_k t_k| \leq 4\mu(Q+2\LL)=\mu_1, \quad k \in [-n, n].
\end{equation}

Consider the sequence $c_k = t_k - b_k$. Note that
$c_{-n} = 0$ by (\ref{fly2}). Estimates (\ref{mu1}) imply that
$|c_{k+1}-A_kc_k|\leq \mu_1$. Hence,
$$
|c_k| \leq ((N+1)^{2n} + (N+1)^{2n-1}+ \dots + 1)\mu_1<1, \quad k
\in [-n, n].
$$
Thus,
\begin{equation}
\label{fly1}
|t_k - b_k| < 1.
\end{equation}
Consider the sequence $\{z_k\in \tm, k \in \Z\}$ defined as follows:

$$
\begin{cases}
z_k = \Delta_k - b_k, &  \mbox{$k \in [-n, n+1]$},\\
z_k = 0, &  \mbox{$k \notin [-n, n+1].$}
\end{cases}
$$

Inequalities (\ref{Add7.1.1}) and (\ref{fly1}) imply estimate
(\ref{1.4}), while equalities (\ref{2.1}) and (\ref{fly2}) imply
relations (\ref{1.5}). Lemma \ref{lem2} is proved.
\end{proof}



\begin{proof}[Proof of Lemma \ref{lem1}] Fix $n > 0$ and consider the sequence
$$
w_k^{(n)} = \begin{cases}  w_k, & \mbox{$k \in [-n, n],$}\\
 0, & \mbox{$|k|>n.$}
\end{cases}
$$

By Lemma \ref{lem2}, there exists a sequence $\{z_k^{(n)}\in \tm,k
\in \Z\}$ such that
\begin{equation}
\label{Text2.2}
|z_k^{(n)}| \leq 8\LL +1,  \quad k \in \Z,
\end{equation}
and
\begin{equation}
\label{4.2}
z_{k+1}^{(n)} = A_k z_k^{(n)} + w_k^{(n)}, \quad k \in [-n, n].
\end{equation}
Passing to a subsequence of $\{z_k^{(n)}\}$, we can find a
sequence $\{v_k\in\tm,k \in \Z\}$ such that
$$
\quad v_k = \lim_{n \to \infty}z_k^{(n)}, \quad k
\in \Z.
$$
(Let us note that we do not assume uniform convergence.) Passing to
the limit in estimates (\ref{Text2.2}) and equalities (\ref{4.2}) as
$n\to\infty$, we get relations (\ref{3.5}). Lemma \ref{lem1} and our
theorem are proved.
\end{proof}

\section*{Acknowledgement}

The research of the second author is supported by NSC (Taiwan)
98-2811-M-002-061.



\end{document}